\documentclass[12pt]{amsart}

\usepackage[OT2, T1]{fontenc}

\usepackage{amscd}
\usepackage{amsmath}
\usepackage{amssymb}
\usepackage{mathrsfs} 			
\usepackage{units}
\usepackage[all]{xy}

\usepackage{algorithm}
\usepackage{algorithmic}

\def\frk{\frak}               

\def\pp{{\frk p}}
\def\Pp{{\frk P}}
\def\qq{{\frk q}}

\def\mm{{\frk m}}

\def\Phi{{\frk n}}
\def\Phi{{\frk N}}
\def\opn#1#2{\def#1{\operatorname{#2}}} 

\opn\projdim{proj\,dim} \opn\injdim{inj\,dim} \opn\rank{rank}
\opn\depth{depth} \opn\sdepth{sdepth} \opn\fdepth{fdepth}
\opn\grade{grade} \opn\height{height} \opn\embdim{emb\,dim}
\opn\codim{codim}  \opn\min{min} \opn\max{max}

\opn\Tr{Tr} \opn\bigrank{big\,rank}
\opn\superheight{superheight}\opn\lcm{lcm}
\opn\trdeg{tr\,deg}
\opn\reg{reg} \opn\lreg{lreg} \opn\ini{in} \opn\lpd{lpd}
\opn\size{size}
\opn\div{div} \opn\Div{Div} \opn\cl{cl} \opn\Cl{Cl}
\opn\Spec{Spec} \opn\Supp{Supp} \opn\supp{supp} \opn\Sing{Sing}
\opn\Ass{Ass} \opn\Min{Min}
\opn\Ann{Ann} \opn\Rad{Rad} \opn\Soc{Soc}
\opn\Im{Im} \opn\Ker{Ker} \opn\Coker{Coker} \opn\Am{Am}
\opn\Hom{Hom} \opn\Tor{Tor} \opn\Ext{Ext} \opn\End{End}
\opn\Aut{Aut} \opn\id{id}  \opn\deg{deg}

\opn\nat{nat}
\opn\pff{pf}
\opn\Pf{Pf} \opn\GL{GL} \opn\SL{SL} \opn\mod{mod} \opn\ord{ord}
\opn\Gin{Gin} \opn\Hilb{Hilb}
\opn\aff{aff} \opn\con{conv} \opn\relint{relint} \opn\st{st}
\opn\lk{lk} \opn\cn{cn} \opn\core{core} \opn\vol{vol}
\opn\link{link} \opn\star{star}
\opn\gr{gr}

\def\pot#1#2{#1[\kern-0.28ex[#2]\kern-0.28ex]}

\opn\dirlim{\underrightarrow{\lim}}
\opn\inivlim{\underleftarrow{\lim}}
\let\tensor=\otimes

\let\to=\rightarrow

\def\Implies{\ifmmode\Longrightarrow \else
        \unskip${}\Longrightarrow{}$\ignorespaces\fi}
\def\implies{\ifmmode\Rightarrow \else
        \unskip${}\Rightarrow{}$\ignorespaces\fi}
\def\iff{\ifmmode\Longleftrightarrow \else
        \unskip${}\Longleftrightarrow{}$\ignorespaces\fi}

\let\:=\colon
\newtheorem{Theorem}{Theorem}[]
\newtheorem{Lemma}[Theorem]{Lemma}
\newtheorem{Corollary}[Theorem]{Corollary}
\newtheorem{Proposition}[Theorem]{Proposition}

\theoremstyle{definition}

\newtheorem{Remark}[Theorem]{Remark}

\newtheoremstyle{subsection-tweak}
   {11pt}
   {3pt}%
   {}
   {}%
   {\bfseries}
   {}%
   {.5em}
   {\thmnumber{\@{#1}{}\@{#2}.}%
    \thmnote{~{\bfseries#3.}}}    

\newcounter{numberingbase}

\theoremstyle{subsection-tweak}
\newtheorem{bpp}[Theorem]{}
\newtheorem{bppt}[numberingbase]{}
\newcommand{\bbpp}{\begin{bpp}}
\newcommand{\eepp}{\end{bpp}}
\newcommand{\bbppt}{\begin{bppt}}
\newcommand{\eeppt}{\end{bppt}}

\theoremstyle{theorem}

\theoremstyle{definition}

\newcommand{\val}{\mathrm{val}}		

\let\epsilon\varepsilon
\let\phi=\varphi
\def\qed{\ifhmode\textqed\fi
      \ifmmode\ifinner\quad\qedsymbol\else\dispqed\fi\fi}
\def\textqed{\unskip\nobreak\penalty50
       \hskip2em\hbox{}\nobreak\hfil\qedsymbol
       \parfillskip=0pt \finalhyphendemerits=0}
\def\dispqed{\rlap{\qquad\qedsymbol}}

\opn\dis{dis}
\def\pnt{{\raise0.5mm\hbox{\large\bf.}}}

\opn\Lex{Lex}

\begin{document}

\title{Extensions of valuation rings containing $\bf Q$ as limits of smooth algebras}

\author{ Dorin Popescu}

\address{Simion Stoilow Institute of Mathematics of the Romanian Academy, Research unit 5, University of Bucharest, P.O. Box 1-764, Bucharest 014700, Romania, Email: {\sf dorin.popescu@imar.ro}}

\begin{abstract} We give a necessary and sufficient condition for  an extension of valuation rings containing $\bf Q$ to be a filtered direct limit of smooth algebras.

\end{abstract}

\maketitle

\section*{Introduction}

Any integral algebraic variety equipped with a dominant morphism from a valuation ring $V$ can be desingularized along $V$ in characteristic zero as it is shown by the Zariski Uniformization Theorem \cite{Z}. This gives that any valuation ring $V$ containing a field $K$ of characteristic zero is a filtered union of regular $K$-subalgebras of $V$ of finite type, or equivalently of smooth $K$-subalgebras of $V$. In particular, any finite system of polynomial equations over $K$ with a solution in $V$ can be "embedded" in a  finite system of polynomial equations over $K$ with a solution in $V$, but for which one can apply the Implicit Function Theorem.

A ring map $A \to A'$ is \emph{ind-smooth} if $A'$ is a filtered direct limit of smooth $A$-algebras. A filtered direct limit (in other words a filtered colimit) is a limit indexed by a small category that is filtered (see \cite[002V]{SP} or \cite[04AX]{SP}). A filtered  union is a filtered direct limit in which all objects are subobjects of the final colimit, so that in particular all the transition arrows are monomorphisms. The Zariski Uniformization Theorem says, in particular,  that $K\to V$ is ind-smooth, if $K\supset {\bf Q}$. This weaker form has a  different proof in \cite{P}. It was mainly a consequence of 
 the following theorem
 (see \cite[Theorem 21]{P}). 
\begin{Theorem}\label{T0} Let $V\subset V'$ be an immediate extension of valuation rings   containing $\bf Q$.
Then $V'$ is ind-smooth over $V$.
\end{Theorem}
An immediate extension of valuation rings is an extension inducing trivial extensions on residue field and group value extensions. 

The goal of this paper is to show the following theorems.

 \begin{Theorem} \label{T1} Let $V\subset V'$ be an extension of valuation rings containing $\bf Q$, $K\subset K'$ its fraction field extension. Then $V'$ is ind-smooth over $V$  if and only if the following statements hold
\begin{enumerate}
\item{} for each $q\in \Spec V$ the ideal $qV'$ is also prime,

\item{}  For any prime ideals $q_1,q_2\in \Spec V$  such that $q_1\subset q_2$ and  $\height(q_2/q_1)=1$ the  extension $V_{q_2}/q_1V_{q_2}\subset V'_{q_2V'}/q_1V'_{q_2V'}$ of  valuation rings is ind-smooth.
\end{enumerate}
\end{Theorem}

 In the next theorem the  necessary and sufficient conditions are  given  in the frame of the value group extension of $V\subset V'$.  

\begin{Theorem} \label{T2} Let $V\subset V'$ be an extension of valuation rings containing $\bf Q$, $\Gamma\subset \Gamma'$ the value group extension of $V\subset V'$  and $\val:\Gamma'\to K'^{*}$ the valuation of $V'$. Then $V'$ is ind-smooth over $V$ if and only if  the following statements hold
\begin{enumerate}

\item for each $q\in \Spec V$ the ideal $qV'$ is prime,

\item  for any prime ideals $q_1,q_2\in \Spec V$ such that $q_1\subset q_2$ and height$(q_2/q_1)=1$  and any $x'\in q_2V'\setminus q_1'$ there exists $x\in V$ such that $\val(x')=\val(x)$,  where $q_1'\in \Spec V'$ is the prime ideal corresponding to the maximal ideal of $V_{q_1}\otimes_V V'$, that is the maximal prime ideal of $V'$ lying on $q_1$.
\end{enumerate}
\end{Theorem}
 These theorems follow from Theorems \ref{t}, \ref{t1} and Corollary \ref{c1}.
 
 Working with solutions in $V'$ of a finite system of polynomial equations over $V$, the above result says when it is possible to reduce to the case of solutions in $V'$ of a finite system of polynomial equations over $V$  for which one can apply the Implicit Function Theorem (see Proposition \ref{P}, this method is  used in the Artin approximation theory). 

\vskip 0.5 cm

\section{Extensions of valuation rings essentially of finite type containing  $\bf Q$}

For a finitely presented ring map $A \to B$, an element $b \in B$ is \emph{standard over $A$} if there exists a presentation $B \cong A[X_1, \dotsc, X_m]/I$ and $f_1, \dots, f_r \in I$ with $r \leq m$ such that $b = b'b''$ with $b' = \det ({(\partial f_i/\partial X_{j}))_{1 \le i,\, j \le r}} \in A[X_1, \dotsc, X_m]$ and a $b'' \in A[X_1, \dotsc, X_m]$ that kills~$I/(f_1, \dotsc, f_r)$ (our standard element is a special power of the standard element from  \cite[Definition, page 9]{S} given in the particular case of the valuation rings, see also \cite[Theorem 4.1]{S}). By abuse we denote also by $b',b''$ the  induced elements of $B$. The radical of the ideal generated by the  elements of $B$ standard over $A$ is $H_{B/A}$, which defines the non-smooth (that is the singular) locus of $B/A$.

We will need the following lemma proved in  \cite[Lemma 7]{P}.

\begin{Lemma} \label{k}
For a commutative diagram of ring morphisms

\xymatrix@R=0pt{
& B \ar[rd] & & && B \ar[dd]^{b \, \mapsto\, a} \ar[rd] & \\
A \ar[rd] \ar[ru] & & V & \mbox{that factors as follows} & A \ar[ru]\ar[rd] & & V/a^3V \\ 
& A' \ar[ru] & & && A'/a^3A' \ar[ru] &
}

\noindent with $B$ finitely presented over $A$, $V$ a valuation ring, an element $b \in B$ that is standard over $A$, and a nonzerodivisor $a \in A'$ that maps to a nonzerodivisor in $V$ that lies in every maximal ideal of $V$, 
there is a smooth $A'$-algebra $S$ such that the original diagram factors as follows:

\hskip 4 cm\xymatrix@R=0pt{
& B \ar[rdd] \ar[rrd] & & \\
A \ar[rd] \ar[ru] & &  & V. \\
& A' \ar[r]  & S \ar[ru] &
}

 \end{Lemma}

Applying the above lemma we need to use as in \cite{P}   the following two lemmas.

\begin{Lemma}(\cite[Lemma 13]{P}) \label{L1}
For ring maps $A \to B \to V$ with $B$ of finite type over $A$, a prime $\Pp \subset V$ with preimage $\pp \subset A$, and a factorization $A \to B \to S' \to V_{\Pp}$ for a finitely presented $A_{\pp}$-algebra $S'$, there are a finitely presented $A$-algebra $S$, an $s \in S$ with $S_{s} \tensor_A A_{\pp} \simeq S'[X, X^{-1}]$, and a  factorization
\[
 A \to B \to S \to V \ \ \mbox{such that} \ \ S \to V_{\Pp} \ \ \mbox{factors as} \ \ S \to S_{s} \tensor_A A_{\pp} \to  V_{\Pp}.
\] 
\end{Lemma}

An extension of valuation rings $V\subset V'$ is {\em dense} if $V'\subset {\hat V}$, where $\hat V$ is the completion of $V$.
\begin{Lemma} (\cite[Proposition 9]{P} \label{L2}
For a  ring $A$, a dense extension of valuation rings  $V\subset V'$, $K$ the  fraction field of $V$,  a ring morphism $A \to V$, a finitely presented $A$-algebra $B$, and maps
\[
A \to B \to V \ \ \mbox{such that} \ \ B \to K \ \ \mbox{factors through some $A$-smooth localization of $B$}
\] 
suppose that
 there exist a smooth $A$-algebra $S'$ and a factorization $
A \to B \to S' \to V'$. Then
 there exist a smooth $A$-algebra $S$ and a factorization $
A \to B \to S \to V$. 
In particular, there exist a smooth $A$-algebra $S$ and a factorization $A \to B \to S \to V$   if  there exist a smooth $A$-algebra $\hat S$ and a factorization $A \to B \to {\hat S} \to {\hat V}$,  ${\hat V}$ being the completion of $V$. 
\end{Lemma}

After these preparations  we may give a
small extension of Theorem \ref{T0} when the residue field extension is not trivial.

\begin{Proposition}\label{p}  Let $V\subset V'$ be an extension of valuation rings  containing $\bf Q$,  $k\subset k'$ its  residue field extension,  $K\subset K'$ its fraction field extension and $\Gamma\subset \Gamma'$ its value group extension. Assume $K'/K$ is a field extension of finite type and $\trdeg K'/K=\trdeg k'/k<\infty$ and one of the following statements hold:
\begin{enumerate}

\item{} the value group extension $\Gamma\subset \Gamma'$ is trivial,

\item for each $q\in \Spec V$ the ideal $qV'$ is also prime and
for any prime ideals $q_1,q_2\in \Spec V$  such that $q_1\subset q_2$ and  $\height(q_2/q_1)=1$ the  extension $V_{q_2}/q_1V_{q_2}\subset V'_{q_2V'}/q_1V'_{q_2V'}$ of  valuation rings is ind-smooth.
\end{enumerate}
 Then $V'/V$ is ind-smooth. 
\end{Proposition}
\begin{proof} Let $\mm$ be the maximal ideal of $V$ and  $x$ be a system of elements of $V'$ inducing a transcendental basis of $k'/k$. Then $x$ is algebraically independent over $V$ and  $W=V'\cap K(x)\cong V[X]_{mV[X]}$, $X$ being some variables,  is ind-smooth over  $V$. Thus $K'$ is a finite extension of $L=K(x)$  the fraction field of $W$ by our hypothesis.     The proof follows the proof of   \cite[Theorem 21]{P}.

Let $E=V[Y]/I$, $Y=(Y_1,\ldots,Y_n)$ be a finitely generated $V$-subalgebra of $V'$ (so finitely presented by \cite[Theorem 4]{Na}) and a  $w:E\to V'$ the inclusion.  By  \cite[Lemma 1.5]{S} it is enough to show that $w$ factors through a smooth $V$-algebra.   By separability  we have  $ w( H_{E/V})\not =0$, let us assume that  $  H_{E/V}V' \supset zV'$ for some $z \in  V'$, $z\not =0$. 
Similarly as in  \cite[Lemma 4]{ZKPP} we may
assume $z=w(z')$, $z'\in E$ and    for some polynomials $ F = (F_1,\ldots,F_r)$
from $I$ that $z'\in NM E$ for some $N\in ((F):I)$ and a $r\times r$-minor $M$ of the Jacobian matrix $(\partial F_i/\partial Y_j)$. Thus we may suppose $z'\in E$ standard over $V$, which is necessary later in order to apply Lemma \ref{k}.

 As $K'/L$ is finite we may assume $z\in W$.  If  $V\subset V'$ is dense  we may apply Lemma \ref{L2}   to see that $w$ factors through a smooth $V$-algebra. Otherwise, the factorization is constructed as follows.  Let $q_1'\in \Spec V'$ be the prime ideal corresponding to the maximal ideal of the fraction ring of $V'$ with respect to the multiplicative system  generated by $z$. Then $q_1'$ is the biggest prime ideal of $V'$ which does not contain $z$. The minimal prime ideal $q'$ of $zV'$ satisfies height$(q'/q_1')=1$.  The idea is to construct using Lemma \ref{k}   a factorization $E\to E'\xrightarrow{w'} V'$ such that $w'(H_{E'/V})\not \subset q'$, where $E'$ is finitely presented over $V$. In several steps we want to find  
  a factorization $E\to E^{(n)}\xrightarrow{w^{(n)}}V'$ such that $w^{(n)}(H_{E^{(n)}/V})V'=V'$ which will be enough.
  The problem is to show that we can find such $E^{(n)}$ in finite steps. For this we  consider a finite partition ${\mathcal P}_i$, $i=1,\ldots,s$ of $\Spec V'$ corresponding to those $q''\in\Spec V'$ which have the same dimension $f_i=f_{q''}\leq f=[K':L]$ of the fraction field extension $L_{q''}\subset K'_{q''}$ of $W/q''\cap W\subset V'/q''$. We will see that to each construction $q_1'$ (associated to a standard element $z$ of some $V$-algebra $E^{(n)}$) change  from one ${\mathcal P}_j$ to another one from  ${\mathcal P}_i$ with $j<i$ and $f_i<f_j$. Finally we arrive in finite steps to the case $f_{q_1'}={\bar f}=[k':k'']$, $k''$ being the residue field of $W$. Note that $f_{q'}\leq f_{q_1'}$.

    Let $x_{q'}$ be a primitive element of the separable finite extension $K'_{q'}/L_{q'}$ and $g_{q'}\in (W/q'\cap W)[X]$ be a primitive polynomial multiple of Irr$(x_{q'},L_{q'})$- the irreducible polynomial of $x_{q'}$ over $L_{q'}$- by a nonzero constant of $L_{q'}$. Note that if $q',\qq'\in {\mathcal P}_i$, $q'\subset \qq'$ then $f_{q'}=f_{\qq'}=f_i$ and $g_{q'}$ remains irreducible over $W/\qq'\cap W$. We may take $x_{\qq'},g_{\qq'}  $ induced by $x_{q'},g_{q'}$.
    Clearly, $f_s={\bar f}$ because $f_{\mm'}=[k':k'']$ for the maximal ideal $\mm'$ of $V'$.
A set  ${\mathcal P}_i$ has a maximum element for inclusion namely $\pp_i'=\cup_{ \qq' \in {\mathcal P}_i}\qq'$. Indeed, $\pp_i'$  is clearly a prime ideal and if  $f_{\pp_i'}<f_i$ then $f_{\qq'}<f_i$ for some $\qq'\in  {\mathcal P}_i$, which is false.

    Assume $q_1'\in {\mathcal P}_j$. If $q_1'\not =\pp_j'$ then $(W/q_1'\cap W)_{\pp_j'\cap W}\subset (V'/q_1')_{\pp_j'}$ is in fact a localization of $(W/q_1'\cap W)[X]/(g_{q_1'})$ because $g_{q_1'}'=\partial g_{q_1'}/\partial X$ corresponds to a unit in $(V'/q_1')_{\pp_j'}$ and so the composite map $E\otimes_VW\to V'\to (V'/q_1')_{\pp_j'}$ induced by $w$ factors through an etale $W/q_1'\cap W$-algebra of the form $((W/q_1'\cap W)[X]/(g_{q_1'})_{g_{q_1'}'h}$ for some $h\in W[X]$. In particular $E\otimes_VW
\to V'\to (V'/(z^3))_{\pp_jV'}$ factors  through an etale $W/(z^3)$-algebra and by Lemma \ref{k} the map  $E\otimes_VW\to V'\to V'_{\pp_j'}$ factors through a smooth $W$-algebra. But $W$ is a localization of a polynomial $V$-algebra and so $w$ factors through a smooth $V$-algebra. Using Lemma \ref{L1}  we see that $w$ factors through a finitely presented $V$-algebra $E'$, let us say through a map $w':E'\to V'$ with $w'(H_{E'/V})\not \subset \pp_j'$. Changing $E$ by $E'$ we see that the new $q'$ belongs to ${\mathcal P}_i$ for some $i>j$. Moreover, the new $q_1'$ belongs also to ${\mathcal P}_{i'}$ for some $i'>j$, because otherwise we get $q_1'=\pp_j'$. 
    
    If $q_1'=\pp_j'$ then $q'\in {\mathcal P}_{j+1}$ and we apply  \cite[Corollary 19]{P} as in \cite[Proposition 20]{P} 
     when (1) holds.  Then     $(W/q_1'\cap W)_{q'}\subset (V'/q_1')_{q'}$ is ind-smooth and  we see that the composite map $E\otimes_VW\to V'\to (V'/q_1')_{q'}$ factors through a
 smooth $W/q_1'\cap W$-algebra. By Lemma \ref{k}  the composite map $E\to V'\to V'_{q'}$ factors through a smooth $V$-algebra and by Lemma \ref{L1} we get that   $w$ factors through a finitely presented $V$-algebra $E'$, let us say through a map $w':E'\to V'$ with $w'(H_{E'/V})\not \subset q'$. Now the new $q_1'$,
 that is the old $q'$, belongs to ${\mathcal P}_{j+1}$.

 Set $q_1=q'_1\cap V$, $q=q'\cap V$ and assume (2) holds. If $q=q_1$, then we see directly that   the composite map $E\to V'\to (V'/q_1')_{q'}$ factors through a
 smooth $(V/q_1)_{q_1}$-algebra by Zariski's Uniformization Theorem (\cite{Z}, see also \cite[Theorem 35]{P}) and using Lemma \ref{k} we see that the composite map $E\to V'\to V'_{q'}$ factors through a
 smooth $V$-algebra. As above by  Lemma \ref{L1}  we see that $w$ factors through a finitely presented $V$-algebra $E'$, let us say through a map $w':E'\to V'$ with $w'(H_{E'/V})\not \subset q'$. Now the new  $q_1'$  that is the old $q'$, belongs to ${\mathcal P}_{j+1}$.

 Suppose $q\not=q_1$. We have $qV'=q'$ because otherwise $qV'\subset q'_1$ and we get $q=q_1$ which is false. Also $q_1V'\subset q'_1$ because otherwise we get $q_1'=q'$,  which is not possible. It follows that height$(q/q_1)=1$ and $q_1V'=q_1'$ since height$(q'/q_1')=1$.
 By (2)
 the  extension $V_q/q_1V_q\subset V'_{q'}/q_1'V'_{q'}$ of  valuation rings is ind-smooth and so the composite map $E\to V'\to (V'/q_1')_{q'}$ factors through a
 smooth $(V/q_1)_{q'}$-algebra. Using Lemma \ref{k} we see that the composite map $E\to V'\to V'_{q'}$ factors through a smooth $V$-algebra. Applying again  Lemma \ref{L1}  we see that   $w$ factors through a finitely presented $V$-algebra $E'$, let us say through a map $w':E'\to V'$ with $w'(H_{E'/V})\not \subset q'$. Now the new $q_1'$,
 that is the old $q'$, belongs to ${\mathcal P}_{j+1}$.

  In some finitely many steps (at most $s$) we arrive to the case when $f_{q_1}={\bar f}$. As before $V'/q_1'$  is a localization of an etale $W/q_1'\cap W$-algebra and so $V'/z^3V'$ is a localization of an etale $W/z^3W$-algebra. Applying Lemma \ref{k} we see that $E\otimes_VW\to V'$ factors through an etale $W$-algebra, that is $w$ factors through a smooth $V$-algebra.
 Now  apply  \cite[Lemma 1.5]{S}.   
\hfill\ \end{proof}

\begin{Remark}\label{r} The above proposition does not work when $k$ has not characteristic  zero. A reason is  that \cite[Corollary 19]{P} does not work in this case as shows for example \cite[Example 3.13]{Po1} inspired by \cite[Sect 9, No 57]{O} (see also \cite[Remark 6.10]{Po1}). Also Zariski's Uniformization Theorem is open when $k$ has not characteristic  zero.
\end{Remark}

\begin{Remark}\label{r1} Let  $V\subset V'$ be an immediate algebraic extension of valuation rings of dimension one and $\hat V $ the completion of $V$. Assume $V\subset {\hat V}$ is separable and transcendental, that is its fraction field extension is separable and contains at least one transcendental element. The example  \cite[Example 3.13]{Po1} shows that $V\subset V'$ can be not dense. But if $V\subset V'$ is ind-smooth then necessarily $V\subset V'$ is dense (see \cite[Theorem 2]{Po}).
\end{Remark}

\begin{Theorem} \label{t} Let $V\subset V'$ be an extension of valuation rings containing $\bf Q$, $K\subset K'$ its fraction fields and $\Gamma \subset \Gamma'$ its value group extension.  If the extension $\Gamma\subset \Gamma'$ is trivial then $V'$ is ind-smooth over $V$. If $K'/K$ is of finite type, but $\Gamma \subset \Gamma'$  is not necessarily trivial, then $V'$ is ind-smooth over $V$  if and only if the following statements hold
\begin{enumerate}
\item{} for each $q\in \Spec V$ the ideal $qV'$ is also prime,

\item{}  For any prime ideals $q_1,q_2\in \Spec V$  such that $q_1\subset q_2$ and  $\height(q_2/q_1)=1$ the  extension $V_{q_2}/q_1V_{q_2}\subset V'_{q_2V'}/q_1V'_{q_2V'}$ of  valuation rings is ind-smooth.
\end{enumerate}
\end{Theorem}

\begin{proof} If $V'$ is ind-smooth over $V$ then $V'/qV'$ is ind-smooth over $V/q$   for every $q\in V$ and so $V'/qV'$ is a domain. Also for $S=V\setminus q_2$ and we have $S^{-1}(V'/q_1V')$  ind-smooth over $S^{-1}(V/q_1)=(V/q_1)_{q_2}$ for  $q_1,q_2\in \Spec V$ with $q_1\subset q_2$. Hence  the necessity in the second statement  holds because $V'_{q_2V'}/q_1V'_{q_2V'}$ is a localization of  $S^{-1}(V'/q_1V')$.

  For the rest  we may also consider  the case when the field extension $K'/K$ is of finite type  in the first statement because we may write $K'$ as a filtered union of subfields $F$ of $K'$ which are finite type field extensions of $K$. The value group extension of $V\subset V'\cap F$ is still trivial and we may replace $V'$ by such $V'\cap F$.

Let   $E $ be  a $V$-algebra of finite presentation, let us say $E \cong
V[Y]/I$, $Y =
(Y_1\ldots, Y_m)$, $I$ being  a finitely generated ideal. Let $w:E\to V'$ be a $V$-morphism. We will show that $w$ factors through a smooth $V$-algebra and the proof ends applying \cite[Lemma 1.5]{S}   as in the proof of the above proposition. We consider $H_{E/V}$, $z$ as in Proposition \ref{p}  and we may assume $z'\in E$ with $z=w(z')$  standard over $V$, which is necessary later to apply Lemma \ref{k}. Let $q',q'_1$ be the adjacent prime ideals of $z$ as in the above proposition. We will show 
that $w$ factors through a finitely presented $V$-algebra $E'$, let us say by a morphism $w':E'\to V'$ such that $w'(H_{E//V})\not \subset q'$.

Let $q''\in \Spec V'$. Set $t=\trdeg K'/K$ and let $t_{q''}$ be the transcendental degree of the fraction field extension $K_{q''}\subset K'_{q''}$ of $V/q''\cap V\subset  V'/q''$. Then $t_{q''}\leq t$. Let $x_{q''}=(x_{q'',1},\ldots,x_{q'',t_{q''}})$ be  some elements from $V'$ which form a transcendental basis of the  field extension  $K_{q''}\subset K'_{q''}$. 

We  consider a finite partition ${\mathcal F}_i$, $i=1,\ldots,s$ of $\Spec V'$ corresponding to those $q''\in\Spec V'$ which have the same transcendental degree $t_i=t_{q''}\leq t$. 
 Set $ \qq'_i=\cup_{ q''\in {\mathcal F}_i}  q''$ and  let $q'\in  {\mathcal F}_i$. Clearly $\qq'_i$ is a prime ideal and if $\qq'_i\not \in {\mathcal F}_i$ then there exists a non zero  polynomial $f\in V[X]$, $X=(X_1,\ldots,X_{t_i}) $ such that $f(x_{q'})\in \qq'_i$. Thus $f(x_{q'})\in q''$ for some $q''\in  {\mathcal F}_i$, $q'\subset q''$, that is $t_{q''}<t_i$, which is false.  So $\qq'_i$  is the greatest prime  
ideal of $V'$ such that $t_{\qq'_i}=t_i$. Given $q_1'\in  
{\mathcal F}_i$ we will construct a $V$-algebra of finite presentation
 $E'$ and a morphism $w':E'\to V'$ such that $w$ factors through $w'$ and changing $E$ by $E'$ the new $q_1$ belongs to a  ${\mathcal F}_j$ for some $j>i$, so $t_j<t_i$. 
Finally we arrive in finite steps to the case $t_{q_1'}={\bar t}=[k':k]$, $k, k'$ being the residue fields of $V,V'$ respectively and it is enough to apply Proposition \ref{p} as we will see at the end of the proof.

 Indeed, assume $q_1'\in {\mathcal F}_i$. If $q_1'\not =\qq_i'$ then apply Proposition \ref{p} to $(V/q'_1\cap V)_{\qq'_i\cap V}\subset (V'/q'_1)_{\qq'_i}$. Then the composite map $E\to V'\to(V'/q')_{\qq'_i}$ induced by $w$ factors through a smooth 
$(V/q'_1\cap V)_{\qq'_i\cap V} $-algebra. By Lemma \ref{k} we see that the composite map $E\to V'\to V'_{\qq'_i}$ factors through a smooth $V$-algebra. Using  Lemma \ref{L1} we see that $w$ factors through a finitely presented $V$-algebra $E'$, let us say by a morphism $w':E'\to V'$ such that $w'(H_{E//V})\not \subset \qq'_i$. Then the new $z$ is not in $\qq'_i $. If the new $q'_1\in {\mathcal F}_i$ then the old $q'_1$ should be $\qq_i'$, which is not the case. 

If $q'_1=\qq'_i$ then we consider the extension   $(V/q'_1\cap V)_{q'\cap V}\subset (V'/q'_1)_{q'}$. In the first case we apply \cite[Corollary 19]{P} as in the above proposition.
 Assume now that (1), (2) hold. Set $q_1=q'_1\cap V$, $q=q'\cap V$. If $q=q_1$, then as in the above proposition we see that   the composite map $E\to V'\to (V'/q_1')_{q'}$ factors through a
 smooth $(V/q_1)_{q_1}$-algebra by Zariski's Uniformization Theorem (\cite{Z}, see also \cite[Theorem 35]{P}) and using Lemma \ref{k} we see that the composite map $E\to V'\to V'_{q'}$ factors through a
 smooth $V$-algebra. As above by  Lemma \ref{L1}  we see that $w$ factors through a finitely presented $V$-algebra $E'$, let us say through a map $w':E'\to V'$ with $w'(H_{E'/V})\not \subset q'$. Now the new  $q_1'$  that is the old $q'$, belongs to ${\mathcal P}_{i+1}$.

 Suppose $q\not=q_1$. As in Proposition \ref{p} we see that height$(q/q_1)=1$ and $q_1V'=q_1'$ since height$(q'/q_1')=1$.
 By (2)
 the  extension $V_q/q_1V_q\subset V'_{q'}/q_1'V'_{q'}$ of  valuation rings is ind-smooth and so the composite map $E\to V'\to (V'/q_1')_{q'}$ factors through a
 smooth $(V/q_1)_{q'}$-algebra. Using Lemma \ref{k}  and  Lemma \ref{L1}  we see as above that   $w$ factors through a finitely presented $V$-algebra $E'$, let us say through a map $w':E'\to V'$ with $w'(H_{E'/V})\not \subset q'$. Now the new $q_1'$,
 that is the old $q'$, belongs to ${\mathcal F}_{i+1}$.

Applying this construction we arrive  in at most $s$ steps  to the case when $t_{q_1'}={\bar t}$, when we apply Proposition \ref{p} to see that the extension  $V/q'_1\cap V\subset V'/q'_1$ is ind-smooth.  Using Lemma \ref{k} we see that $w$ factors through a smooth $V$-algebra
 Now apply \cite[Theorem 1.5]{S}. 
\hfill\ \end{proof}

\begin{Corollary}\label{c} Let $V \subset V'$ be an extension of valuation rings containing $\bf Q$, with
$\mm$, $\mm' = \mm V'$ their maximal ideals and $\dim V = 1$. Let $q\in \Spec V'$ be such that $q\cap V=0$ and
height$(\mm'/q) = 1$. Assume that the extension $V \subset V'/q$ has the same value group.
Then $V'$ is ind-smooth over $V$.
\end{Corollary}
For the proof apply Theorem \ref{t} to see that $V'/q$ is ind-smooth over $V$ which is enough using  Zariski's Uniformization Theorem and Lemmas \ref{k} and \ref{L1}.

\vskip 0.5 cm

\section{Extensions of valuation rings containing $\bf Q$.}

We start reminding a property of ind-smooth morphisms given in \cite[Lemma 28]{P}.

\begin{Lemma}\label{0} Let $V \subset V'$ be an extension of valuation rings which is ind-smooth. Then $\Omega_{V'/V}$, that
is $ H_0(V, V',V')$ in terms of Andre-Quillen homology, is a flat $V'$-module and $H_1(V, V',V')
= 0$.
\end{Lemma}

The following two lemmas  extend closely \cite[Lemmas 29, 30]{P}.
\begin{Lemma}\label{1}

Let $V\subset V'$ be an extension of valuation rings  with $\dim V=1$,  $\Gamma$  the value group of $V$, $\mm$ the maximal ideal of $V$ and $K$ the fraction field of $V$. Let $q'\in \Spec V'$ be the prime ideal corresponding to the maximal ideal of the valuation ring $K\otimes_V V'$ and $\Gamma''$ the value group of the valuation ring $V''=V'/q'$.   Assume that $\mm V'$ is the maximal ideal of $V'$. Then $\dim V''=1$ and if the valuation group extension  $\Gamma\subset \Gamma''$ of $V\subset V''$  has a non trivial  torsion
then the extension $V\subset V'$ is not ind-smooth.
\end{Lemma}

\begin{proof} Let $\gamma\in \Gamma''\setminus \Gamma$ be such that $n\gamma\in \Gamma$ for some positive integer $n$. Choose an element $x\in V'$ such that $\val(x)=\gamma$ in $\Gamma''$. Then $x^n\equiv zt''$ modulo $q'$ for some $z\in V$ and an unit $t''\in V'$. Thus $x^n= zt$ modulo $q''$ for some  unit $t\in V'$  and the system $S$ of polynomials $X^n=zT$, $TT'=1$ over $V$ has a solution in $V'$. If $V'$ is ind-smooth over $V$ then $S$ has a solution in a smooth $V$-algebra and so one $({\tilde x},{\tilde t},{\tilde t}')$ in the completion of $V$, which is Henselian. But then $\gamma =\val(z)/n=\val({\tilde x})$ must be in $\Gamma$ which is false.
\hfill\ \end{proof}

\begin{Lemma}\label{2}

Let $V\subset V'$ be an extension of valuation rings  with $\dim V=1$ and containing $\bf Q$. Let $\mm$ be the maximal ideal of $V$, $K$ the fraction field of $V$,  $q'\in \Spec V'$  the prime ideal corresponding to the maximal ideal of the valuation ring $K\otimes_V V'$ and $\Gamma''$ the value group of the valuation ring $V''=V'/q'$. Assume that  $\mm V'$ is the maximal ideal of $V'$ (so $\dim V''=1$) and the value group  $\Gamma\subset {\bf R}$ of $V$ is   dense in $\bf R$. Also assume that the value group $\Gamma''\subset {\bf R}$ of $V''$ contains  an element $\gamma\not \in \Gamma$ inducing in $\Gamma''/\Gamma$ an element without torsion.
Then the extension $V\subset V'$ is not ind-smooth.
\end{Lemma}

\begin{proof}
Let $x\in V'$ be such that $\val(x)=\gamma$ and
   $V'_0=V'\cap K(x)$. 

We will show that $\Omega_{V'_0/V}$ has torsion.  The  proof  idea is from \cite[Lemma 7.2]{Po0} (see also \cite[Lemma 30]{P}).  We consider as in the  quoted lemmas two real sequences $(u_i)$, $(v_i)$ from $\Gamma$  which converge in $\bf R$ to $\gamma_1$ and such that
 $u_{j+1}>u_j, v_j>v_{j+1}$,  $u_j<\gamma<v_j$  for all $j$. Let   $a_j$, $b_j$ be in $V$ with values $u_j$, resp. $v_j$.
Note that $V'_0$ is a filtered  union  of localizations $C_j$ of $V[Z_j,Z'_j]/(Z_jZ'_j-(b_j/a_j))\cong V[z_j,z'_j]$, where   $z_j=x/a_j$ and $z'_j=b_j/x$ in $V'$, the map $C_j\to C_{j+1}$ being given by $Z_j\to (a_{j+1}/a_j)Z_{j+1}$, $Z'_j\to (b_j/b_{j+1})Z'_{j+1}$.

Let  $f_j:C_{j+1}\otimes_{C_j} \Omega_{C_j/V}\to  \Omega_{C_{j+1}/V'}$ be the map given by $d z_j\to (a_{j+1}/a_j)d z_{j+1}$, $d z'_j\to (b_j/b_{j+1})dz'_{j+1}$. Then $K\otimes_V f_j$ is injective. Indeed, an element from \\
$\Ker (K\otimes_V f_j)$ induced by $w=\alpha\otimes dz_j+\beta \otimes dz'_j$, $\alpha, \beta\in C_{j+1}$ must go by $f_j$ in
 $$\alpha (a_{j+1}/a_j)dz_{j+1}+\beta (b_j/b_{j+1})dz'_{j+1}\in <z'_{j+1}dz_{j+1}+z_{j+1}dz'_{j+1}>$$
 
\noindent in $C_{j+1}dz_{j+1}\oplus C_{j+1}dz'_{j+1}$. So $\alpha (a_{j+1}/a_j)=\mu z'_{j+1}$ and $\beta (b_{j}/b_{j+1})=\mu z_{j+1}$ for some $\mu \in C_{j+1}$. It follows that 
$\alpha=\mu(b_{j+1}/b_j)(a_j/a_{j+1})z'_j$ and $\beta=\mu(b_{j+1}/b_j)(a_j/a_{j+1})z_j.$
Note that  $\eta_{j+1}=((b_{j+1}/b_j)(a_j/a_{j+1}))^{-1}\in V$ and so 
$$\eta_{j+1}w=\mu( z'_jdz_j+ z_jdz'_j)\in <z'_jdz_j+z_jdz'_j>,$$
 which shows our claim.

We may assume that $\val( z_j) \leq \val( z'_j)$ and so $z'_j=t_jz_j$  for some $ t_j\in V_0'$. Thus there exists $j' \geq j$ such that $t_j\in C_{j'}$. We  have $dz_j, dz'_j$ in $\Omega_{C_{j'}/V}$ and
$z_jw'_j=0$, for $w'_j=dz'_j+t_jdz_j$. So $ C_{j'}\otimes_{C_j} \Omega_{C_j/V}$ is not torsion free. Moreover $w'_j$ and $z_j$ are not killed by multiplication with non zero elements of $V$. Thus $K\otimes_V(C_{j'}\otimes_{C_j} \Omega_{C_j/V})$ is not torsion free. Since $K\otimes_V f_j$ is injective we see that $K\otimes_V \Omega_{V'_0/V}$ - the limit of $K\otimes_V \Omega_{C_j/V}$ is not torsion free. In particular, $\Omega_{V'_0/V}$ is not torsion free.

 Now, in the Jacobi-Zariski sequence (see \cite[Theorem 3.3]{S}  applied to $V\to V'_0\to V'$
  $$H_1(V,V',V')\to V'\otimes_V \Omega_{V'_0/V}\to \Omega_{V'/V}\to  \Omega_{V'/V'_0}\to 0$$

\noindent we have $H_1(V,V',V')=0$ and $ \Omega_{V'/V}$ flat because $V'/V$ is ind-smooth (see Lemma \ref{0}) with the help of \cite[Theorem 3.4]{S}).
It follows  that  $V'\otimes_{V'_0} \Omega_{V'_0/V}$ has no torsion and so $ \Omega_{V'_0/V}$ has also no torsion because $V'_0\subset V'$ is flat. This is not possible because as above it has torsion.  Thus $V'$ is not ind-smooth over $V$.
\hfill\ \end{proof}

\begin{Proposition}\label{p1}
Let $V\subset V'$ be an extension of valuation rings containing $\bf Q$ with $\dim V=1$,   $\mm$ the maximal ideal of $V$ and $K$ the fraction field of $V$. Let $q'\in \Spec V'$ be the prime ideal corresponding to the maximal ideal of the valuation ring $K\otimes_V V'$.   Assume that $\mm V'$ is the maximal ideal of $V'$. Then the extension $V\subset V'$ is ind-smooth if and only if 

$(*)$ \ \ For all $x'\in V'\setminus q'$  there exists $x\in V$ with $\val(x')=\val(x)$. 
\end{Proposition}
\begin{proof} Assume that $(*)$ does not hold and let $x'\in \mm V'\setminus q'$ such that $\val(x')\not \in \Gamma$, where $\Gamma$ is the value group of $V$. If $V$ is a DVR then $V''$ is a DVR too and $(*)$ holds because $\mm V''$ is the maximal ideal of $V''$. So $V$ is not a DVR and we may assume $\Gamma\subset {\bf R}$ is dense  in $\bf R$.
 Apply Lemma \ref{2} (we preserve the denotations from there) and  we see that $V\subset V'$ is not ind-smooth when $\val(x')$ induces an element  without torsion in $\Gamma''/\Gamma$. 
If $\val(x')$ induces an element  with torsion in $\Gamma''/\Gamma$ then apply Lemma \ref{1}. So the necessity holds.

Now suppose that $(*)$ holds. Then $\Gamma=\Gamma''$. By \cite[Proposition 20]{P} we see that $V''$ is ind-smooth over $V$. Let $B$ be a $V$-algebra of finite presentation and $v:B\to V'$ a $V$-morphism. As in the proof of \cite[Proposition 20]{P} we see that the smooth locus $H_{B/V}$ satisfies $v(H_{B/V})\not =0$ by separability. Moreover, using the Zariski Uniformization Theorem \cite{Z} applied for $K\subset K\otimes_V V'$  and \cite[Lemma 13]{P} we see that $v$ factors through a $V$-algebra $E\subset V'$ of finite presentation  such that  $H_{E/V}V'\not \subset q'$. As   $V''$ is ind-smooth over $V$ it follows by Lemma \ref{k}  that the inclusion $E\subset V'$ factors through a smooth $V$-algebra. Using \cite[Lemma 1.5]{S} we see that $V'$ is ind-smmoth over $ V$.
\hfill\ \end{proof}

\begin{Theorem}\label{t1}
Let $V\subset V'$ be an extension of valuation rings containing $\bf Q$ with $\dim V=1$,   $\mm$ the maximal ideal of $V$ and $K$ the fraction field of $V$. Let $q'\in \Spec V'$ be the prime ideal corresponding to the maximal ideal of the valuation ring $K\otimes_V V'$.   Assume that $\mm V'$ is a prime ideal of $V'$. Then the extension $V\subset V'$ is ind-smooth if and only if 

$(**)$ \ \ For all $x'\in \mm V'\setminus q'$  there exists $x\in V$ with $\val(x')=\val(x)$. 
\end{Theorem}

\begin{proof} Assume that $(**)$ holds. By the above proposition $V'_{\mm V'}$ is ind-smooth over $V$. Let $B$ be a $V$-algebra of finite presentation and $v:B\to V'$ a $V$-morphism. Then $v$ factors through  a $V$-algebra $E$ of finite presentation, let us say by a $V$-morphism $w:E\to V'$, such that $w(H_{E/V})\not \subset \mm V'$. By the   Zariski Uniformization Theorem \cite{Z} applied for $V/\mm\subset  V'/\mm V'$ we see that $V'/\mm V'$ is ind-smooth over $V/\mm$. Using  Lemma \ref{k} we conclude  that $w$ factors through a smooth $V$-algebra and we may apply \cite[Lemma 1.5]{S}.

On the other hand, if $V'$ is ind-smooth over $V$ then $V'_{\mm V'}$ is ind-smooth over $V$ and we apply the above proposition. 
\hfill\ \end{proof}

\begin{Corollary} \label{c1} Let $V\subset V'$ be an extension of valuation rings containing $\bf Q$ such that for each $q\in \Spec V$ the ideal $qV'$ is prime. The following statements are equivalent:
\begin{enumerate}
\item $V'$ is ind-smooth over $V$,

\item for any prime ideals $q_1,q_2\in \Spec V$ such that $q_1\subset q_2$ and height$(q_2/q_1)=1$ the extension $(V/q_1)_{q_2}\to (V'/q_1 V')_{q_2 V'}$ is ind-smooth,

\item  for any prime ideals $q_1,q_2\in \Spec V$ such that $q_1\subset q_2$ and height$(q_2/q_1)=1$  and any $x'\in q_2V'\setminus q_1'$ there exists $x\in V$ such that $\val(x')=\val(x)$, where $q_1'\in \Spec V'$ is the prime ideal corresponding to the maximal ideal of $V_{q_1}\otimes_V V'$, that is the maximal prime ideal of $V'$ lying on $q_1$.
\end{enumerate}
\end{Corollary}
\begin{proof} By the above theorem we have (2) and (3) equivalent. Note that (1) implies (2). Assume that (3) holds. Then for any field extensions $K\subset L\subset K'$ we see that  the condition (3)(and so (2)) holds for the extension $V\subset V_L=V'\cap L$. In particular for the case when $L/K$ is of finite type. Using Theorem \ref{t}, we see that $V_L$ is smooth over $V$ when $L/K$ is of finite type. But then $V'$ is ind-smooth over $V$ because $V'$ is the filtered union of such $V_L$, when $L/K$ is of finite type.
\hfill\ \end{proof}

Note that this corollary extends Theorem \ref{T0} because an immediate extension of valuation rings containing $\bf Q$ satisfies (3).

\begin{Proposition} \label{P} Let $V\subset V'$ be an extension of valuation rings, which is ind-smooth. Assume that $V$ is Henselian. Then any finite
system of polynomials over $V$, which has a solution in $V'$, has also one in $V$. In particular, a Henselian valuation ring containing $\bf Q$  has the property of  Artin approximation (this is extended in \cite[Corollary 1.2.1]{MB})  
\end{Proposition}
\begin{proof} Let $f$  be a finite
system of polynomials over $V$ and $y'\in V'$ a solution of $f$. As $V'$ is a filtered limit of smooth $V$-algebras there exists a solution $y''$ of $f$ in a smooth 
$V$-algebra $C$. Clearly, there exists a retraction $h:C\to V$ of $V\subset C$ by the Implicit Function Theorem because $V$ is Henselian and $h(y'')$ is a solution of $f$ in $V$.

Let $\hat V$ be the completion of $V$. For the second statement note that  $V\subset {\hat V}$ is immediate ans so ind-smooth by Theorem \ref{T0}.
\hfill\ \end{proof}

\end{document}